\documentclass[12pt]{amsart}
\usepackage{amsmath, amsthm}
\usepackage{amssymb}
\usepackage{amsfonts}
\usepackage{abstract}
\usepackage{yhmath}
\usepackage{epsfig}
\usepackage{indentfirst}
\usepackage[usenames,dvipsnames]{pstricks}
\usepackage{pst-grad}
\usepackage{pst-plot}
\usepackage{mathrsfs}
\usepackage[all]{xy}
\usepackage{ragged2e}
\usepackage{bbm}

\title[GK dimensions of simple modules over twisted group algebras $k \ast A$]{Gelfand-Kirillov dimensions of simple modules over twisted group algebras $k \ast A$}

\author[Ashish Gupta]{Ashish Gupta$^1$}
\address{$^1$Department of Mathematics, 
Ramakrishna Mission Vivekananda Educational and Research Institute (RKMVERI), Belur Math, Howrah, Box: 711202, 
West Bengal, India.}
\email{a0gupt@gmail.com}

\author[Umamaheswaran Arunachalam]{Umamaheswaran Arunachalam$^2$}
\address{$^2$Department of Mathematics,
              Harish-Chandra Research Institute (HRI),
              Chhatnag Road, Jhunsi,
              Prayagraj (Allahabad), Box: 211019, Utter Pradesh, India}
\email{ruthreswaran@gmail.com}


\newtheoremstyle{theorem}
  {10pt}		  
  {10pt}  
  {\sl}  
  {\parindent}     
  {\bf}  
  {. }    
  { }    
  {}     
\theoremstyle{theorem}
\newtheorem{theorem}{Theorem}
\newtheorem{corollary}[theorem]{Corollary}

\newtheorem{lemma}[theorem]{Lemma}

\newtheorem{proposition}[theorem]{Proposition}
\newtheorem{remark}[theorem]{Remark}

\newtheorem{example}[theorem]{Example}
\newtheorem{question}[theorem]{Question}

\newtheorem*{theorem_1}{Theorem 1}
\newtheorem*{theorem_2}{Theorem 2}

\DeclareMathOperator{\gk}{\mathscr{GK}}

\DeclareMathOperator{\aut}{Aut}

\DeclareMathOperator{\rk}{rk}

\DeclareMathOperator{\End}{End}

\newtheoremstyle{defi}
  {10pt}		  
  {10pt}  
  {\rm}  
  {\parindent}     
  {\bf}  
  {. }    
  { }    
  {}     
\theoremstyle{defi}
\newtheorem{definition}[theorem]{Definition}

\newtheoremstyle{defi}
{10pt} 
{10pt} 
{\rm} 
{\parindent} 
{\bf} 
{. } 
{ } 
{} 
\theoremstyle{defi}

\begin{document}

\maketitle

\begin{abstract}
For the $n$-dimensional multiparameter quantum torus algebra $\Lambda_{\mathfrak q}$ over a field $k$ defined by a multiplicatively antisymmetric matrix $\mathfrak q = (q_{ij})$ we show that in the case when the torsion-free rank of the subgroup of $k^\times$ generated by the $q_{ij}$ is large enough there is a characteristic set of values (possibly with gaps) from $0$ to $n$ that can occur as the Gelfand--Kirillov dimension of simple modules. The special case when $\mathrm{K}.\dim(\Lambda_{\mathfrak q}) = n - 1$ and $\Lambda_{\mathfrak q}$ is simple studied in  A.~Gupta, {\it $\uppercase{\mbox{GK}}$-dimensions of simple modules over $K[X^{\pm 1}, \sigma]$}, Comm. Algebra, {\bf 41(7)} (2013), 2593--2597 is considered without assuming simplicity and it is shown that a dichotomy still holds for the GK dimension of simple modules.
\end{abstract}

\section{Introduction}
\label{sect-1}
Let $k$ be a field and let $k^\times$ denote the group $k \setminus \{0\}$. Let $\mathfrak q$ be an $n \times n$ multiplicatively antisymmetric matrix with entries in $k^\times$. This means that the entries (called \emph{multi-parameters}) $q_{ij}$ of $\mathfrak q$ satisfy $q_{ii} = 1$ and $q_{ji} = q_{ij}^{-1}.$
The \emph{$n$-dimensional quantum torus} $\Lambda_{\mathfrak q}$ is the associative algebra generated over the field $k$ by the variables $X_1,\cdots, X_n$ together with their inverses subject to the relations 
\begin{align}\label{fund_relations}
X_i X_j &= q_{ij}X_j X_i,    \ \ \ \ \ \ \  \ \ \ \forall 1 \le i, j \le n . \ \ \ 
\end{align}

     These algebras play an important role in noncommutative geometry~\cite{MN:1991} and have found applications in the representation theory of torsion-free nilpotent \\ groups~\cite{BS1:2000}. For example, if $H$ is a finitely generated and torsion-free nilpotent group of class two with center $\zeta H$ then the central localization $kH(k \zeta H \setminus \{0\})^{-1}$ is an algebra of this type.  

While many different aspects including irreducible and projective modules~\cite{AM1:1995,AM2:1995,AM3:1996,AM4:1996,AM6:1998,AM7:1999,AM8:1999} automorphisms and derivations~\cite{AC:1992,AM5:1997,AW:2001,NB:2008,OP:1995} Krull and global dimensions~\cite{AG:1992,BS1:2000,MP:1988, AG3:2016} etc. of these algebras have been studied in the recent past (a survey appears in \cite{AM9:2008}) the focus of the present article is Gelfand--Kirillov dimensions of simple modules over the  quantum tori.

As is well-known the Gelfand--Kirilov dimension (GK dimension) is an important invariant of non-commutative algebras and their modules. For modules over the algebras we are studying the GK dimension is related to another important invariant for modules, namely, the Krull dimension~\cite{MP:1988}. For further details on the GK dimenion we refer the reader to the excellent references \cite{KL:2000} and \cite{MR:1987}.


 

The structure of $\Lambda_{\mathfrak q}$ and its modules can vary substantially depending on the matrix of multiparameters $\mathfrak q$. For example, if $\mathfrak q$ has entries that generate a subgroup of $k^\times$ of maximal possible (torsion-free) rank, namely, $\frac{n(n -1)}{2}$ then $\lambda_{\mathfrak q}$ is a simple hereditary noetherian domain~\cite{MP:1988} whereas in the case when this same subgroup has rank zero, it is an algebra finitely generated  over its center. Moreover the matrix $\mathfrak q$ is not uniquely determined by the isomorphism class of $\lambda_{\mathfrak q}$~\cite{MP:1988}.

The situation thus calls for having an additional invariant which is convenient enough to work with and at the same time well-suited to our goal of studying the simple modules of quantum torus algebras. This invariant is provided by the Krull dimension of non-commutative rings which in our case coincides with another dimensional invariant, namely, the global dimension~\cite{MP:1988}. 

The simple modules over the quantum torus in the case when it has Krull dimension one were studied in \cite{MP:1988} and also in \cite{AM3:1996, AM7:1999} (for a generic quantum torus). In \cite{AG1:2014} and \cite{AG2:2016} focus was shifted to the case where the Krull dimension is one less than the maximum possible, that is, $n - 1$.
In \cite{AG1:2014} a dichotomy for the GK dimensions of simple $\Lambda_{\mathfrak q}$-modules was established assuming that $\Lambda_{\mathfrak q}$ has Krull dimension $n - 1$ and is itself simple. 
In the current article we show that this dichotomy remains intact even when this simplicity assumption is lifted:

\begin{theorem_1}
Let $\Lambda_{\mathfrak q}$ be an $n$-dimensional quantum torus algebra with Krull dimension $n - 1$. For the GK dimension of a simple $\Lambda_{\mathfrak q}$-module $M$ the following dichotomy holds   
\[\gk (M) = 1, \ \ \   \mathrm{or} \ \ \    \gk (M)  = \gk( \Lambda_{\mathfrak q}) - \gk(\mathcal Z(\Lambda_{\mathfrak q})) - 1,\]
where $\mathcal Z(\Lambda_{\mathfrak q})$ denotes the center of $\Lambda_{\mathfrak q}$.
\end{theorem_1}



In carrying out a study of simple modules over quantum polynomials we often find it useful to impose some kind of an independence assumption on the multiparameters $q_{ij}$. For example, in the consideration of simple modules in \cite{AM3:1996, AM7:1999} it is assumed that the multiparameters are in general position, that is, generate a subgroup in $k^\times$ of maximal rank.

Evidently, we may present an $n$-dimensional quantum torus $\Lambda_{\mathfrak q}$ as a skew Laurent polynomial ring over the subring  $\Lambda'$ generated by the variables $X_1^{\pm 1}, \cdots, X_{n -1}^{\pm 1}$. Thus $\Lambda_{\mathfrak q}  = \Lambda'[X_n^{\pm 1}; \sigma]$ where $\sigma$ is a scalar automorphism of $\Lambda'$ given by $X_i \mapsto p_iX_i$ for $p_i \in k^\times$. In the following we will denote the set of GK dimensions of simple modules over a given quantum torus  algebra $\Lambda_{\mathfrak q}$ by $\mathscr V(\Lambda_{\mathfrak q})$. Our next theorem expresses $\mathscr V(\Lambda_{\mathfrak q})$ in terms of $\mathscr V(\Lambda')$ assuming an independence condition for the multiparameters (see Theorem 2).

\begin{definition}\label{lmda-grp}
(i) The $\lambda$-group $\mathscr G(\Lambda_{\mathfrak q})$ of a quantum torus algebra $\Lambda_{\mathfrak q}$ is defined as the subgroup of $k^\times$ generated by the multi-parameters $q_{ij}$. \\
 \noindent (ii) Given a scalar automorphism $\sigma$ of $\Lambda_{\mathfrak q}$ we define $\mathscr H_\sigma$ as  the subgroup of $k^\times$ generated by the scalars $p_i$ ($i = 1, \cdots, n$). 
\end{definition}

We can now state our second theorem.


\begin{theorem_2}
Let $\Lambda_{\mathfrak{q}}$ be an $n$-dimensional quantum torus algebra and 
consider the skew-Laurent extension \[ \Lambda^\ast_{\mathfrak  q, \sigma} = \Lambda_{\mathfrak q}[Y^{\pm 1} ; \sigma], \] 
where $\sigma \in \aut (\Lambda_{\mathfrak{q}})$ is a scalar automorphism defined by $\sigma(X_i) = p_i X_i$.
Assume that the subgroups  $\mathscr G(\Lambda_{\mathfrak q})$ and  $\mathscr H_\sigma$ of $k^\times$ as in Definition \ref{lmda-grp} intersect trivially. 
Let $\mathscr V (\Lambda_\mathfrak q)$ be the (finite) set of GK dimensions of simple $\Lambda_{\mathfrak q}$-modules and similarly $\mathscr V(\Lambda^\ast_{\mathfrak  q, \sigma})$ the set of GK dimensions of simple $\Lambda^\ast_{\mathfrak  q, \sigma}$-modules. Then 
 \[ \mathscr V(\Lambda^\ast_{\mathfrak  q, \sigma}) \subseteq 
  \{\rk(\mathscr H_\sigma), \cdots, n\} \cup ( \mathscr V(\Lambda_\mathfrak q) + 1).  \]
\end{theorem_2} Here $\mathscr V(\Lambda_\mathfrak q) + 1$ stands for the set $\{u + 1 \mid u \in \mathscr V(\Lambda_\mathfrak q)\}$.
\begin{remark}
Theorem 2 does not necessarily mean that for each value $d$ in the set \[ \{\rk(\mathscr H_\sigma), \cdots, n\} \cup ( \mathscr V(\Lambda_\mathfrak q) + 1) \] there is a simple $\Lambda^\ast_{\mathfrak  q, \sigma}$-module with GK dimension $d$. Rather it points to the general fact that in the case when sufficiently many of the multiparameters are independent there are certain characteristic values in the set $1, \cdots, n$ (possibly with gaps) that can occur as GK dimensions of simple modules over the quantum torus algebra (Example \ref{eg1}).
\end{remark}
 We provide some examples illustrating our two theorems at the end of Section \ref{proof_thm_2}.
This paper completes the story, so to speak, for the GK dimension of simple modules over the quantum torus $\Lambda_{\mathfrak q}$ in the case when $\mathrm{K}.\dim(\Lambda_{\mathfrak q}) = n - 1$ initiated in \cite{AG1:2014} and continued in \cite{AG2:2016}. In \cite{MP:1988} the corresponding problem for $\mathrm{K}.\dim(\Lambda_{\mathfrak q}) = 1$ was considered and answered.
Our work thus naturally leads to the following question.

\begin{question}
Let $\Lambda_{\mathfrak q}$ be an n-dimensional quantum torus algebra with Krull dimension either $2$ or $n - 2$. Does there exist a non-holonomic simple $\Lambda_{\mathfrak q}$-module with GK dimension $n - 1$? If so, how can such a simple module be constructed. What are the possible values of the GK dimension of simple $\Lambda_{\mathfrak q}$-modules. 
\end{question}

\section{The $n$-dimensional quantum torus}\label{qtorus}
In this section we recall some known facts concerning the quantum torus algebras that we shall be needing in the development of our results to follow.

\subsection{Twisted group algebra structure}

As already noted the algebras $\Lambda_{\mathfrak q}$ have the structure of a twisted group algebra $k \ast A$ of a free abelian group of rank $n$ over $k$. We briefly recall this kind of a structure and refer the reader to \cite{PM:1985} for further details. For a given group $G$ a $k$-algebra $R$ is said to be twisted group algebra of $G$ over $k$ if $R$ contains as $k$-basis a copy $\overline{G}: = \{ \bar{g} \mid g \in G \} $ of $G$ such that the multiplication in $R$ satisfies: 
\begin{equation}
\label{t-mult}
\bar g_1 \bar g_2 = \gamma(g_1, g_2)\overline{g_1g_2}
\end{equation}
where $\gamma : G \times G \rightarrow k^{\times}$. The associativity of multiplication means that the \emph{twisting} function $\gamma$ satisfies:
\[ \gamma(x,y)\gamma(xy, z) = \gamma(y, z)\gamma(x, yz) \]
which is the 2-cocycle condition in terms of group cohomology (each twisted group algebra $k \ast G$ arises from an element in $H^2(G, k^\times)$).
For a subgroup $H$ of $G$ the $k$-linear span of $\overline{H} := \{\bar {h} \mid h \in H \}$,  is a sub-algebra of $R$ which is a twisted group algebra of $H$ over $k$ with the defining cocycle being the restriction of $\gamma$ to $H \times H$. This subalgebra will be denoted as $k \ast H$.
In the case $G$ is abelian it is known that the center of $k \ast G$ is of the form $k \ast Z$ for a suitable subgroup $Z \le G$ (e.g \cite[Lemma 1.1]{OP:1995}).

\subsection{Commutative (sub-)twisted group algebras and Krull dimension} \label{Kdim_results}
In the case of the quantum tori $\Lambda_q = k \ast A$  the subgroups $B$ for which the subalgebra $k \ast B$ is commutative play an important role. For example, the following fact  was conjectured in \cite{MP:1988} and shown in \cite{BG1:2000}. 

\begin{theorem}\label{sup-ranks-K.dim}
Given a quantum torus algebra $k \ast A$ the supremum of the ranks of subgroups $B \le A$ such that the corresponding (sub-) twisted group algebra $k \ast B$ is commutative equals both the Krull and the global dimensions of the algebra $k \ast A$.
\end{theorem}
We also note the following important fact is a rewording of  \cite[Theorem 3]{BG1:2000}. 

\begin{proposition}
Suppose that a quantum torus algebra $k \ast A$ has a finitely generated module with GK dimension $m$ then $A$ has a subgroup $B$ with rank equal to $\rk(A) - m$ for which the corresponding (sub-)twisted group algebra $k \ast B$ is commutative.
\end{proposition}

The following interesting corollary is a clear consequence of combining the last two results recalling from \cite{MP:1988} that the GK dimension of a finitely generated $k \ast A$-module is a non-negative integer.

\begin{corollary}\label{hol_numb}
The GK dimension of a finitely generated $k \ast A$-module $M$ satisfies 
\[ \gk(M) \ge \rk(A) - \mathrm{K.dim}(k \ast A), \]
where $\mathrm{K.dim}(k \ast A)$ stands for the Krull dimension of the algebra $k \ast A$.
\end{corollary}

\subsection{Ore subsets and the subgroups of finite index in $A$}
\label{Ore_subsets-fin_ind_sbgrp}
Using the fact that $A \cong \mathbb Z^n$ is an ordered group (with the lexicographic order) it is not difficult to show that $k \ast A$ is a domain and each 
 unit in this ring has the form $\mu \bar a$ for $\mu \in k^{\times}$ and $a \in A$. 
 
 Let $M$ be a finitely generated $k \ast A$-module. If $A_0$ is a subgroup of $A$ with finite index then it is not difficult to see that $M$ is a finitely generated $k \ast A_0$-module. Moreover the algebra $k \ast A$ is a \emph{finite normalizing extension} of its subalgebra $k \ast A_0$, that is, it is generated over the latter by a finite number of normal elements namely the images in $k \ast A$ of a transversal $T$ of $A_0$ in $A$. We recall that in a ring  $T$ an element $t \in T$ normalizes a subring $S$ if $tS = St$.  
 
 It is well-known fact (e.g., \cite{MP:1988}) that for a subgroup $B \le A$ the non-zero elements of the (sub-) twisted group algebra $k \ast B$ is an Ore subset in $k \ast A$. 
 \subsection{Crossed products resulting from Ore localization}
A more general structure than a twisted group algebra is a \emph{crossed product} $R \ast G$ of a group $G$ over a ring $R$. Here the ground ring $R$ need not be a field and the scalars in $R$ need not be central in $R \ast G$. As in the case of twisted group algebras a copy $\bar G$ of $G$ contained in $R \ast G$ is an $R$-module basis and multiplication of the basis elements is defined exactly as in $(\ref{t-mult})$ above. However, as already remarked the scalars in $R$ need not commute with the basis elements $\bar g$  but the relation $\bar g r= \sigma_g(r)\bar g$ holds for $r \in R$ and $\sigma_g \in \aut(R)$.      
We refer the interested reader to the text \cite{PM1:1989}
for further details concerning crossed products. Besides being generalizations of twisted group algebras, crossed products also arise as suitable localizations of the former rings. It is in this latter form that a crossed product will arise in our paper. For example, the Ore localization $(k \ast A)(k \ast B \setminus \{0\})^{-1}$ is a crossed product $D \ast A/B$, where $B \le A$ and $D$R denotes the quotient division ring of the noetherian domain $k \ast B$.





\section{The GK dimension of finitely generated $\Lambda_{\mathfrak q}$-modules}
\label{BG-dim--GK-dim}

The GK dimension is a particularly well-behaved dimension for the finitely generated modules over the algebras we are studying. One reason for this is that the Hilbert-Samuel machinery which works for almost commutative algebras can be adapted for our class of algebras as well (Section 5 of \cite{MP:1988}). For example, we have the following: 

\begin{proposition}\label{dim_exact}
Let \[ 0 \rightarrow L \rightarrow M \rightarrow N \rightarrow 0 \] 
be an exact sequence of $\Lambda_{\mathfrak q}$-modules. Then \[\gk(M) = \max\{\gk(L), \gk(N)\}\] 
\end{proposition}

\begin{proof}
It follows from \cite[Lemma ~5.5]{MP:1988}.
\end{proof}

In \cite{BG1:2000} the finitely generated modules over crossed products of a free abelian group of finite rank over a division ring $D$ were studied with group-theoretic applications in mind.     
A dimension for finitely generated modules which was shown to coincide with the Gelfand--Kirillov dimension (measured relative to $D$) was introduced and studied. 
We aim to employ this dimension in our investigation of  simple modules over the algebras we are considering here which are special cases of the aforementioned crossed products.
We state below its definition and some key properties which were established in \cite{BG1:2000}. This dimension is used in conjunction with an appropriate notion of a critical module to be discussed below. 

\begin{definition}[\cite{BG1:2000}]\label{Br-Gr_Defn}
Let $M$ be a finitely generated $\Lambda_{\mathfrak{q}}$-module. The dimension $\dim M$ of $M$ is the maximum $r$, where $0 \leq r \leq n$ so that for some subset $\mathcal I :=  \{i_1, i_2, \cdots, i_r\}$ of the indexing set $\{1, \cdots, n\}$ the module $M$ is not torsion as $\Lambda_{\mathfrak q, \mathcal I}$-module where
$\Lambda_{\mathfrak q, \mathcal I}$ denotes the subalgebra of $\Lambda_{\mathfrak q}$ generated by the variables $X_i$ for $i \in \mathcal I$ and their inverses.     
\end{definition}

\begin{remark}
\label{key_remrk_1}
In \cite{BG1:2000} it was shown that the dimension $\dim M$ of $M$ in the sense of the last definition coincides with the GK dimension of $M$. 
\end{remark}

The following facts from \cite{BG1:2000} which we state for the algebras $\Lambda_{\mathfrak q}$ 
were shown for more general crossed products.
 
\begin{lemma}
\label{embeds_inf_free_module}
Let $M$ be a finitely generated $\Lambda_{\mathfrak{q}}$-module  with GK dimension $d$ 
and let $\Lambda_1$ be the subalgebra of $\Lambda_{\mathfrak{q}}$ generated by the variables $\{ X_{i_i}, X_{i_2}, \cdots, X_{i_d} \}$  and their inverses.  
Then $M$ cannot not embed a $\Lambda_1$-submodule which is free of infinite rank.
\end{lemma}

\begin{proof}
Noting Remark \ref{key_remrk_1} this follows from \cite[Lemma 2.3]{BG1:2000}.
\end{proof}

\begin{definition}
A nonzero $\Lambda_{\mathfrak q}$-module $N$ is said  to be \emph{critical} if for each non-zero submodule $N$ of $M$  
\[ \gk(M/N) < \gk(M). \]
\end{definition}

\begin{proposition} \cite[Proposition ~2.5]{BG1:2000}\label{CRIT_MOD_EXIST}
Every non-zero $\Lambda_{\mathfrak{q}}$-module contains a finitely generated critical submodule.  
\end{proposition}



\section{The dichotomy result}


\begin{theorem_1}
Let $\Lambda_{\mathfrak q}$ be an n-dimensional quantum torus algebra with Krull dimension $n - 1.$ For the GK dimension of a simple $\Lambda_{\mathfrak q}$-module $M$ the following dichotomy holds   
\[\gk (M) = 1, \ \ \   \mathrm{or} \ \ \    \gk (M)  = \gk( \Lambda_{\mathfrak q}) - \gk(\mathcal Z(\Lambda_{\mathfrak q})) - 1,\]
where $\mathcal Z(\Lambda_{\mathfrak q})$ denotes the center of $\Lambda_{\mathfrak q}$.
\end{theorem_1}

\begin{proof}
As in Section \ref{qtorus} we may write  $\Lambda_\mathfrak q$ as a twisted group algebra  
$\Lambda_\mathfrak q : = k \ast A$ for a free abelian group $A$ with rank $n$. Moreover, we let $\bar {a}$ stand for the image of $a \in A$ in $\Lambda_\mathfrak q$.   
It is not difficult to see that $\mathcal Z(\Lambda_{\mathfrak q})$ has the form 
 $k \ast Z$  for a suitable subgroup $Z$ of $A$(e.g, Lemma 1.1 of \cite{OP:1995}).
As \[ \gk(k \ast \mathbb Z^l) = l \] (e.g., \cite{MP:1988}), the latter alternative in the assertion of the theorem then reads
\[ \gk (M)  = \rk(A) - \rk(Z) - 1. \]
Let $P$ be the annihilator of $M$ in $k \ast Z.$ Clearly $P$ is a prime ideal of $k \ast Z$. The action of $k \ast Z$ on $M$ gives an embedding \[ (k \ast Z)/P \hookrightarrow \End_{\Lambda_{\mathfrak q}}(M).\] 

It is well-known (e.g.,Proposition \cite[9.4.21]{MP:1988}) that the quantum torus algebra $k \ast A$ satisfies Nullstellensatz and in particular $\End_{\Lambda_{\mathfrak q}}(M)$ is algebraic over $k$. Hence so is $(k \ast Z)/P$.  As a commutative affine algebraic domain is a field~\cite{LHR:2006} it follows that $P$ is a maximal ideal of  $k \ast Z$. 

Set $K  = (k \ast Z)/P$ and $Q = P\Lambda_{\mathfrak q}$. Clearly, $M$ is a simple $\Lambda_{\mathfrak q}/Q$-module.
By \cite[Chapter 1, Lemmas 1.3 and 1.4]{PM1:1989} the $k$-algebra $\Lambda_{\mathfrak q}/Q$ is a twisted group algebra $K \ast A/Z$ of $A/Z$ over $K$ with a transversal $T$ for $Z$ in $A$ yielding a $K$-basis as the set $\{ \bar t + Q \mid t \in T\}$. Moreover, the elements $\zeta + Q$, where $\zeta \in k \ast Z$ constitute a copy of $K$ in $\Lambda_{\mathfrak q}/Q$. We note that the group-theoretic commutator $[\bar{t_1} + Q , \bar{t_2} + Q]$ with values in the 
unit group of  $\Lambda_{\mathfrak q}/Q$ for any $t_1, t_2 \in T$ satisfies \begin{equation}
\label{quotient}
[\bar{t_1} + Q , \bar{t_2} + Q] = [\bar{t_1}, \bar{t_2}] + Q \in k^\times + Q. 
\end{equation}

Now, in view of Proposition 5.1(c) of \cite{KL:2000} we have 
\[ \gk\mbox{-}\dim_{k \ast A}(M) = \gk\mbox{-}\dim_{K \ast A/Z}(M). \]

In the last equation in both LHS and RHS the GK dimension is being measured relative to $k$. Since $K$ is finitely generated and algebraic over $k$ therefore $[K : k] < \infty$ and in view of \cite[Lemma 2(ii)]{WU:1991}, it suffices to determine the possible values of the GK dimension of the simple $K \ast A/Z$-module $M$ measured relative to $K$.   

Our main point in passing to the algebra $K \ast A/Z$ is that as a $K$-algebra it is central, that is, has center $K$. 
Indeed, as we already saw in Section \ref{qtorus}, the center of $K \ast A/Z$ is of the form $K \ast Y$ for a subgroup $Y$ of $A/Z$.
If the image $\bar t + Q$ of  some coset $tZ \in Y$ centralized all elements of $K \ast A/Z$  then using (\ref{quotient}) we have,  
\[ [\bar t, \bar {t_1}] + Q = 1 + Q.    \ \ \ \ \ \ \  \forall t_1 \in T.  \]
As $[\bar t, \bar{t_1}] \in k^\times$, it follows that $[\bar t, \bar{t_1}] = 1$. 
Thus, if $\bar t + Q$ lies in the center of $K \ast A/Z$ then $\bar t$ must be in the center of $k \ast Z$ of $k \ast A$.  Since $t \in T$, this is possible only if $t = 1$.   

As already noted in Section \ref{qtorus} by the theorem of Brookes \cite{BS1:2000} the  dimension of $k \ast A$ equals the supremum of the ranks of the subgroups $B \le A$ such that the subalgebra $k \ast B$ is commutative. In the present situation this means the existence of a subgroup $B$ of $A$ with rank $n - 1$ such that $k \ast B$ is commutative.
  
In passing to $K \ast A/Z$ although the center becomes equal to the base field a small difficulty appears, namely, that $A/Z$ need not be torsion-free. To overcome this we may replace $A$ by a subgroup $A_0$ of finite index such that $A_0/Z$ is torsion-free. Then $B_0: =  A_0 \cap B$ is a subgroup of $A_0$ with rank $n - 1$ and clearly $k \ast B_0$ is a commutative sub algebra of $k \ast A_0$. Evidently, $k \ast B_0Z$ is commutative and therefore by the preceding paragraph $\rk(B_0Z) = \rk(B_0)$. 
Replacing $B_0$ by $B_0Z$ if necessary we may assume that $B_0 \ge Z$. 
In view of (\ref{quotient}) the subalgebra $K \ast B_0/Z$ of $K \ast A_0/Z$ is commutative. We obviously have \[ \rk(B_0/Z) = n - 1 - \rk(Z) = \rk(A_0/Z) - 1 = \rk(A/Z) - 1.\] 

The last equation means that $K \ast A_0/Z$ is a $n - \rk(Z)$-dimensional quantum torus over $K$ with (Krull or global) dimension $n  - \rk(Z) - 1$. As a module over the sub-algebra $K \ast A_0/Z$ the simple $K \ast A/Z$-module $M$ need not remain simple.
However as $K \ast A/Z$ is a finite normalizing extension of $K \ast A_0/Z$ (Section \ref{qtorus}) the $K \ast A/Z$-module $M$ decomposes as a finite direct sum of simple $K \ast A_0/Z$-modules (e.g., Exercise 15A.3 \cite{LHR:2008}). We thus have 
\[ M = N_1 \oplus N_2 \oplus \cdots \oplus N_s \]
as $K \ast A_0/Z$-modules. 
By Lemma 2.7 of \cite{BG1:2000},    
\[\gk\mbox{-}\dim_{K \ast A/Z}(M) = \gk\mbox{-}\dim_{K \ast A_0/Z}(M).\]
Moreover the GK dimension of a finite direct sum of modules is the maximum of the GK dimensions of the summands (Proposition 5.1 of \cite{KL:2000}). In view of these remarks, to establish the theorem it suffices to show that 
if $F$ is field and  $F \ast \mathbb Z^r$ is a twisted group algebra with center $F$ and dimension equal to $r - 1$ then for any simple $F \ast \mathbb Z^r$-module $N$ the following dichotomy holds 
\begin{equation}
\gk(N) = 1 \ \ \ \mathrm{or} \ \ \  \gk(N) = r - 1.
\end{equation}
But this is precisely the content of \cite[Theorem 2.1]{AG1:2014}. 
 
\end{proof}

\subsection{Proof of Theorem 2}
\label{proof_thm_2}

 \begin{theorem_2}
Let $\Lambda_{\mathfrak{q}}$ be an $n$-dimensional quantum torus algebra and 
consider the skew-Laurent extension \[ \Lambda^\ast_{\mathfrak  q, \sigma} = \Lambda_{\mathfrak q}[Y^{\pm 1} ; \sigma], \] 
where $\sigma \in \aut (\Lambda_{\mathfrak{q}})$ is a scalar automorphism defined by $\sigma(X_i) = p_iX_i$.
Assume that the subgroups  $\mathscr G(\Lambda_{\mathfrak q})$ and  $\mathscr H_\sigma$ of $k^\times$ as in Definition \ref{lmda-grp} intersect trivially. 
Let $\mathscr V (\Lambda_\mathfrak q)$ be the (finite) set of GK dimensions of simple $\Lambda_{\mathfrak q}$-modules and similarly $\mathscr V(\Lambda^\ast_{\mathfrak  q, \sigma})$ the set of GK dimensions of simple $\Lambda^\ast_{\mathfrak  q, \sigma}$-modules. Then 
 \[ \mathscr V(\Lambda^\ast_{\mathfrak  q, \sigma}) \subseteq 
 \{\rk(\mathscr H_\sigma), \cdots, n\} \cup ( \mathscr V(\Lambda_\mathfrak q) + 1). \]  
\end{theorem_2}
Here $\mathscr V(\Lambda_\mathfrak q) + 1$ stands for the set $\{u + 1 \mid u \in \mathscr V(\Lambda_\mathfrak q)\}$.

\begin{proof}  
 
We write $\Lambda^\ast$ for ${\Lambda^\ast}_{\mathfrak{q}, \sigma}$ and $\Lambda$ for $\Lambda_{\mathfrak{q}}$.  Noting Proposition \ref{CRIT_MOD_EXIST}, we let $N$ be a finitely generated critical $\Lambda$-submodule of $M.$ 
Consider the $\Lambda^\ast$-submodule $N'$ of $M$ generated by $N$:  
\begin{equation}\label{sum_of_conj}
N^\prime : = N\Lambda^\ast = \displaystyle\sum_{i \in \mathbb{Z}} N Y^i.
\end{equation}

Since $N$ is assumed to be critical, therefore $N \ne 0$ and $N^\prime = M$. 
If the sum in (\ref{sum_of_conj}) is direct, then $N \Lambda^\ast \cong N \displaystyle\otimes_{\Lambda} \Lambda^\ast$ and Lemma 2.4 of \cite{BG1:2000} gives    
\begin{equation}
\label{induction_cse}
\gk(M) = \gk(N) + 1.\end{equation} 
Moreover, since the (left) $\Lambda$-module $\Lambda^\ast$ is free, it is faithfully flat, and it follows from this that $N$ must be a simple $\Lambda$-module.  Noting equation (\ref{induction_cse}) we thus obtain 
\begin{equation}\label{conc_1}
 \gk(M) \in \mathscr V (\Lambda_\mathfrak q) + 1. 
\end{equation}
Clearly the assertion of the theorem holds true in this case.
 We are thus left with the possibility where the sum $\displaystyle \sum_{i \in \mathbb{Z}} NY^i$ fails to be direct. We know from Lemma 2.4 of \cite{BG1:2000} that in this case 
\begin{equation}\label{sum_not_dir}
\gk(N) = \gk(M)
\end{equation}
recalling that the dimension being referred to in this same lemma coincides with the GK dimension measured relative to the ground field $k$. Let $d$ denote the common value in the last equation. 

In view of Definition \ref{Br-Gr_Defn} and the succeeding remark, 
there is a subset $I = \{i_1, \cdots, i_d\}$ of the indexing set $\{1, \cdots, n \}$ such that 
$N$ (and therefore $M$) is not $S: = \Lambda_{\mathfrak q}(I)\setminus \{0\}$-torsion, where $\Lambda_{\mathfrak q}(I)$ denotes the subalgebra of $\Lambda$ generated by the indeterminates 
$X_i^{\pm 1}$ for $i \in I$.

Consider the Ore localization $\Lambda^\ast S^{-1}$. By the definition of $S$ we know that $M$ is not $ S$-torsion and hence the corresponding localization $MS^{-1}$ is nonzero. Note that the ring $\Lambda^\ast  S^{-1}$ contains the quotient division ring $\mathscr D: = \Lambda_{\mathfrak q}(I)S^{-1}$.
Using Lemma \ref{embeds_inf_free_module} it is easily seen that $MS^{-1}$ is a finite dimensional $\mathscr D$-space. Set  $s = \dim_{\mathscr D} M S^{-1}$. 

Define 
\[ J : =  \{ 1, \cdots , n \} \setminus I \] and let $\mathscr G(I)$ denote the following subgroup of $\mathscr G(\Lambda_\mathfrak q)$:
\[ \mathscr G(I) := \langle   q_{kl}  \mid  k, l \in I  \rangle. \] 
Next, for each $j \in J$, we  define 
\[ \mathscr G(I, j) := \langle   q_{kj}  \mid  k \in  I  \rangle. \] 
We  will also need to refer to the subgroup $\mathscr H( I)$ of $\mathscr H_\sigma$ defined as follows:
\[ \mathscr H( I) : =  \langle p_i \mid i \in  I \rangle. \] 

As in the hypothesis of the theorem,   
\[ YX_j =  p_jX_jY ,  \  \   \ \   \ \  \ \ \   \forall j  \in J.\]
This is precisely the situation of Section 3.9 of \cite{MP:1988} (where the `unlocalized' generators are the generators indexed by $J$) and exactly as in that section the following dependence relations must hold:  
\begin{equation}
p_j^s \in \langle \mathscr G(I), \mathscr G( I, j),  \mathscr H( I) \rangle \ \ \ \ \ \  \forall j \in J.   
\end{equation}
By the hypothesis in the theorem, $\mathscr G(\Lambda_{\mathfrak q}) \cap \mathscr H_\sigma = 1$ and therefore $p_j^s \in \mathscr H(I)$.
But this means that \[ \rk(\mathscr H_\sigma) \le \lvert I \lvert = d = \gk(M). \]  
It remains to show that $\gk(M) < n + 1$. To this end we suppose that $\gk(M) = n + 1$. 
Since $\Lambda^\ast$ is a twisted group algebra it follows by Definition \ref{Br-Gr_Defn} and Remark \ref{key_remrk_1} that $M$ embeds a copy of the right regular module $\Lambda^\ast$. As $M$ is simple hence $M \cong \Lambda^\ast$. But this means that $\Lambda^\ast$ is a division ring which is clearly not true.
Our proof is now complete.
\end{proof}

An example illustrating the theorem is in place:

\begin{example}\label{eg1}
Let $F$ be a field and let $\mathfrak q \in \mathrm{M}_n(F)$ be a multiplicatively antisymmetric matrix given by
\begin{equation*} \mathfrak q := 
\begin{pmatrix} 
 1 & 1 & 1 & q_1\\
 1& 1& 1 & q_2 \\
 1& 1& 1 & q_3 \\
q_1^{-1} & q_2^{-1} &   q_3^{-1} & 1 
\end{pmatrix}
\end{equation*}
where it is assumed that 
\begin{equation}\label{maxm_rank_Lmbda_grp}
     \rk(Q) = 3  \ \ \ \ \ \ \mbox{for} \ \ \ \ \ Q := \langle q_1, q_2, q_3 \rangle.
\end{equation}
Let $\Lambda: = \Lambda_{\mathfrak q}$ be the quantum torus algebra defined by the matrix $\mathfrak q$. Let $\sigma$ be the scalar automorphism of $\Lambda$ defined by the vector $(p_1, p_2, p_3, p_4)$ where it is assumed that \begin{equation} \label{rk_sgma}
\rk(P) = 4 \ \ \ \ \ \ \mbox{for} \ \ \ \ \ P := \langle p_1, p_2, p_3, p_4 \rangle.
\end{equation}
We also assume that $P \cap Q = 1$ (These conditions can be realized for example in the field of rational numbers $\mathbb Q$ by choosing distinct primes for the multiparameters $q_i$ and $p_j$). 
Then the quantum torus algebra $\Lambda^\ast$ in Theorem 2 is defined by the matrix $\mathfrak q^\ast$ which is given by
\begin{equation*} \mathfrak q^\ast :=  
\begin{pmatrix} 
 1 & 1 & 1 & q_1 & p_1^{-1} \\
 1& 1& 1 & q_2 & p_2^{-1} \\
 1& 1& 1 & q_3 & p_3^{-1}\\
q_1^{-1} & q_2^{-1} &  q_3^{-1} & 1 & p_4^{-1} \\
p_1 & p_2 & p_3 & p_4 & 1
\end{pmatrix}
\end{equation*}
The algebra $\Lambda$ is satisfies the hypothesis of Theorem 1. Moreover the condition (\ref{maxm_rank_Lmbda_grp}) implies that $\mathcal Z(\Lambda) = k$ (\cite[Proposition 1.3]{MP:1988}).
It now follows from Theorem 1 that $\mathscr V(\Lambda) = \{1, 3 \}$. 
Noting (\ref{rk_sgma}) Theorem 2 now gives 
\[\mathscr V(\Lambda^\ast) =  \{ 4 \}  \cup \{ 2, 4 \} = \{ 2, 4 \}.  \]
Thus for a simple $\Lambda^\ast$-module $M$, \[ \gk(M) \not \in \{1,3,5\}. \]
\end{example}


\section*{Acknowledgments}
The first author was supported by NBHM grant 2/48(14)/2015/NBHM(R.P.)/R \& D II/4147 and the second author was supported by HRI, Prayagraj (Allahabad)postdoctoral fellowship grant M170563VF.

\end{document}